\documentclass[11pt]{article} \textheight23.775cm \textwidth16cm
\hoffset=-1.9cm \voffset=-2.6cm \usepackage{amsmath, amscd, amssymb,
  amsthm, verbatim}

\newtheorem{theorem}{Theorem}[section]

\newtheorem{corollary}[theorem]{Corollary}
\newtheorem{lemma}[theorem]{Lemma}

\theoremstyle{definition} \newtheorem{definition}[theorem]{Definition}
\newtheorem{remark}[theorem]{Remark}
\newtheorem{example}[theorem]{Example}
\newtheorem{notation}[theorem]{Notation}

 \newcommand{\A}{\mathcal{A}}
 
\newcommand{\E}{\mathbb{E}} \newcommand{\PP}{\mathbb{P}}
 \newcommand{\Q}{\mathbb{Q}}
\newcommand{\R}{\mathbb{R}} 
 
 \newcommand{\X}{\tilde{X}}
 
\newcommand{\eps}{\varepsilon} \newcommand{\parr}{\slash\!\slash}

\newcommand\newdot{{\kern.8pt\cdot\kern.8pt}}
\newcommand\bull{{\hbox{\bf .}}}

\DeclareMathOperator{\arcsinh}{arcsinh}
\DeclareMathOperator{\Ric}{Ric} 
\DeclareMathOperator{\Cut}{Cut} \DeclareMathOperator{\tr}{tr}
 \DeclareMathOperator{\Hess}{Hess}
\DeclareMathOperator{\m}{m} 
\DeclareMathOperator{\Id}{id}

\DeclareMathOperator{\Riem}{Riem} \DeclareMathOperator{\defo}{def}
\DeclareMathOperator{\trace}{trace} \DeclareMathOperator{\dist}{dist}

\begin{document}

\title{A stochastic approach to the harmonic map heat flow\\ on
  manifolds with time-dependent Riemannian metric}

\author{Hongxin Guo\footnote{School of Mathematics and Information
    Science, Wenzhou University, Wenzhou, Zhejiang 325035,
    China. E-mail address: \texttt{guo@wzu.edu.cn}}, Robert
  Philipowski\footnote{Corresponding author; Unit\'e de Recherche en
    Math\'ematiques, Universit\'e du Luxembourg, 6,~rue Richard
    Coudenhove-Kalergi, 1359 Luxembourg, Grand Duchy of
    Luxembourg. E-mail address:\ \texttt{robert.philipowski @uni.lu}},
  Anton Thalmaier\footnote{Unit\'e de Recherche en Math\'ematiques,
    Universit\'e du Luxembourg, 6, rue Richard Coudenhove-Kalergi,
    1359 Luxembourg, Grand Duchy of Luxembourg. E-mail address:\
    \texttt{anton.thalmaier@uni.lu}}}

\date{}

\renewcommand{\theequation}{\arabic{section}.\arabic{equation}}
\maketitle

\begin{abstract}
  \noindent
  We first prove stochastic representation formulae for space-time
  harmonic mappings defined on manifolds with evolving Riemannian
  metric. We then apply these formulae to derive Liouville type
  theorems under appropriate curvature conditions.  Space-time
  harmonic mappings which are defined globally in time correspond to
  ancient solutions to the harmonic map heat flow. As corollaries, we
  establish triviality of such ancient solutions in a variety of
  different situations.
\end{abstract}

\section{Introduction}
A smooth mapping $u\colon M\to N$ between Riemannian manifolds $(M,g)$
and $(N,h)$ is said to be {\em harmonic} if its tension field
$\Delta^{g,h} u\equiv\trace\nabla du$ vanishes, see e.g.\
\cite{Eells-Sampson:64, linwang}. Since harmonic maps are
characterized by the property that they map $M$-valued Brownian
motions to $N$-valued martingales (see e.g.\
\cite[Satz~7.157~(ii)]{hackenbroch}), it is natural to study them
using stochastic methods, and this has been done in a number of
papers, e.g.~\cite{arnaudonlithalmaier, Cranston_Kendall_Kifer:1996,
  kendall88, kendall98, Picard:2002, thalmaierwang}. In particular,
stochastic representation formulae for the differential of harmonic
maps have turned out to be a powerful tool to prove Liouville
theorems, i.e.\ theorems stating that harmonic maps in a certain class
of maps and under certain topological or geometric constraints are
necessarily constant~\cite{thalmaierwang}.

Due to Perelman's proof of the geometrization and hence the Poincar\'e
conjecture using Ricci flow \cite{perelman1, perelman2, perelman3},
there is now a strong interest in studying manifolds $M$ with
time-dependent geometry. In such a context, the notion of harmonic map
turns out to be no longer appropriate; however, it is natural to study
\emph{space-time harmonic maps} which by time reversal provide
solutions to the \emph{harmonic map heat flow} (or \emph{nonlinear
  heat equation}), see e.g.~\cite{Li:2013, mueller,
  Michael_Bradford_Williams}.

The behaviour of (positive) solutions to the linear heat equation
under Ricci flow has been intensively studied during the last decade,
e.g.~\cite{Chow_et_al:2008, morgantian:2007, Mueller:2006}. It is
clear from the static case that in the nonlinear situation under Ricci
flow also the geometry of the target space will naturally play a
crucial role, see e.g.~\cite{linwang, Struwe:88}.

Building on our previous work on martingales on manifolds with
time-dependent connection~\cite{martingales}, we establish stochastic
representation formulae for space-time harmonic maps and solutions to
the harmonic map heat flow defined on a manifold with time-dependent
metric. We then apply these formulae to prove Liouville theorems for
space-time harmonic maps and ancient solutions to the harmonic map
heat flow under appropriate curvature conditions.

\section{Stochastic representation formulae} 
Let $M$ be a differentiable manifold equipped with a smooth family
$g(t)$ of Riemannian metrics ($t\in (T_0, T]$ with $T_0 < T$), and let
$(N,h)$ be a Riemannian manifold. Let $u\colon (T_0, T] \times M \to
N$ be a solution to the harmonic map heat flow
\begin{equation}\label{Eq:heatequation}
  \frac{\partial u}{\partial t} =\frac{1}{2}\Delta^{g(t),h} u.
\end{equation}
Here $\Delta^{g(t),h}u:=\trace\nabla du\in\Gamma(u^*TN)$ denotes the
tension field of $u$ with respect to $g(t)$ and $h$. Recall that, for
fixed $t$, we have the differential
$du(t,\newdot)\in\Gamma(T^*M,u(t,\newdot)^*TN)$ and the Hessian
$(\nabla du)(t,\newdot)\in\Gamma(T^*M\otimes T^*M,u(t,\newdot)^*TN)$
which provides for each $(t,x)$ a bilinear map
$$(\nabla du)(t,x)\colon T_xM\times T_xM\to T_{u(t,x)}N.$$
The trace of this bilinear map gives the tension
$(\Delta^{g(t),h}u)(t,x)\in T_{u(t,x)}N$ of $u$ at $(t,x)$.  Any
solution to Eq.~\eqref{Eq:heatequation} with $T_0=-\infty$ is called
\emph{ancient solution} to the harmonic map heat flow.

\begin{remark} 
  Let $u\colon (T_0, T]\times M\to N$ be a solution to the harmonic
  map heat flow~\eqref{Eq:heatequation} and let
  \[
  \hat u(t,\newdot)=u(T-t,\newdot)\quad\text{and}\quad\hat
  g(t)=g(T-t,\newdot)
  \]
  be defined by time reversal. Then the mapping
  \[
  \hat u\colon[0,T-T_0)\times M\to N
  \]
  is space-time harmonic with respect to the family $\hat g(t)$ of
  metrics, i.e.
  \begin{equation*}
    \frac{\partial\hat u}{\partial t} +\frac{1}{2}\Delta^{\hat g(t),h}\hat u=0.
  \end{equation*}
  In particular, any ancient solution to the harmonic map heat
  flow~\eqref{Eq:heatequation} gives rise to a space-time harmonic map
  defined on $\R_+\times M$, and vice versa. For this reason, in the
  sequel, we formulate our results for space-time harmonic maps;
  however the statements immediately apply to solutions of the
  harmonic map heat flow by time reversal.
\end{remark}

From now on, let $(g(t))_{t\geq0}$ be a smooth family of Riemannian
metrics on $M$ and $(N,h)$ be a Riemannian manifold.  Let $u\colon
[0,\infty)\times M\to N$ be a space-time harmonic map in the sense
that
\begin{equation}\label{Eq:spacetimeharmonic}
  \frac{\partial u}{\partial t} +\frac{1}{2}\Delta^{g(t),h} u=0.
\end{equation}

\begin{notation}
  Fixing a point $x\in M$, let $(X_t)_{t\geq 0}$ be a $g(t)$-Brownian
  motion on $M$ \cite{ACT, Coulibaly, kuwadaphilipowski,
    kuwadaphilipowski2, paeng} starting at $x$, and consider the image
  process $\X_t := u(t, X_t)$ taking values in the target
  manifold~$N$. As in \cite[Theorem~9.3 and Remark~9.4]{martingales}
  let $\Theta_{0,t}: T_x M\to T_{X_t}M$ be the damped parallel
  transport along $X$, defined by the covariant equation
  \begin{equation*}
    d\left(\left(\parr_{0,t}^{\Riem}\right)^{-1}\Theta_{0,t}\right) = -\frac{1}{2}\left(\parr_{0,t}^{\Riem}\right)^{-1}\left(-
      \frac{\partial g}{\partial t} + \Ric^{g(t)}\right)^\#\Theta_{0,t} \,dt,\quad\Theta_{0,0} = \Id_{T_x M},
  \end{equation*}
  where $\parr_{0,t}^{\Riem}: T_xM\to T_{X_t}M$ is the
  Riemann-parallel transport along $X$, see
  \cite[Definition~3.3]{martingales}. Similarly, in terms of the
  Riemann curvature tensor $\tilde R$ on $N$, let
  \[
  \tilde\Theta_{0,t}\colon T_{\X_0}N\to T_{\X_t}N
  \]
  be the damped parallel transport along $\X$, defined by the
  covariant equation
  \begin{align*}
    d\left( \parr_{0,t}^{-1}\tilde\Theta_{0,t}\right) & =  -\frac{1}{2}\parr_{0,t}^{-1}\tilde R(\tilde\Theta_{0,t}, d\X_t)d\X_t\\
    & = -\frac{1}{2}\sum_{i=1}^m\parr_{0,t}^{-1} \tilde
    R\left(\tilde\Theta_{0,t}, du(t, X_t)\xi_t^i\right) du(t, X_t)
    \xi_t^i \, dt,\quad \tilde \Theta_{0,0} =\Id_{T_{\tilde X_0}N},
  \end{align*}
  where $m=\dim M$ and $(\xi_t^1,\ldots,\xi_t^m)$ is an adapted
  $g(t)$-orthonormal basis of $T_{X_t}M$:
  \[
  \xi_t^i=\parr_{0,t}^{\Riem}e_i
  \]
  for some fixed orthonormal basis $(e_1,\ldots,e_m)$ of $T_xM$.
\end{notation}

In \cite[Proposition~9.6]{martingales}, we obtained the following
theorem which is crucial for all subsequent results.

\begin{theorem}\label{martheat}
  For each $v\in T_xM$ the $T_{u(0,x)}N$-valued process
  \begin{equation*}
    \tilde\Theta_{0,t}^{-1}\, du(t, X_t)\,\Theta_{0,t}\, v, \quad t\geq0,
  \end{equation*}
  is a local martingale.
\end{theorem}

The following corollary extends
{\cite[Theorem~5.5]{arnaudonthalmaier}} from the case of a fixed
metric to the case of evolving metrics.

\begin{corollary}
  \label{repform}
  Let $T>0$. Assume that there is a constant $\alpha\in\R$ such that
  \begin{equation*}
    \Ric_{g(t)}-\frac{\partial g}{\partial t} \geq\alpha\quad\text{on $[0,T]\times M$,}
  \end{equation*}
  that the sectional curvatures of $N$ are bounded from above
  and that the differential $du$ of $u$ is uniformly bounded on
  $[0,T]\times M$.  Then, for each $0\leq t\leq T$,
  \begin{equation}\label{repr}
    du(0, x) = \E\left[\tilde\Theta_{0,t}^{-1}\, du(t, X_t)\,\Theta_{0,t}\right].
  \end{equation}
\end{corollary}

\begin{proof}
  Under the above assumptions the local martingale of
  Theorem~\ref{martheat} is bounded on the time interval $[0,T]$ and
  hence a true martingale. The claim follows by taking expectations.
\end{proof}

Corollary~\ref{repform} implies a first Liouville theorem under the
assumptions that the metric on $M$ evolves under uniformly strict
backward super Ricci flow and that the curvature of $N$ is
non-positive.

\begin{theorem}[cf.\ {\cite[Proposition~9.6]{martingales}}]
  Let $M$ be connected. Suppose that there is a constant $\alpha > 0$
  such that
  \begin{equation*}
    \Ric_{g(t)}-\frac{\partial g}{\partial t} \geq\alpha
  \end{equation*}
  (uniformly strict backward super Ricci flow), that the sectional
  curvatures of $N$ are non-positive
  and that the differential of $u$ is uniformly bounded.  Then $u$ is
  constant.
\end{theorem}

\begin{proof}
  By Corollary~\ref{repform}
  \begin{equation*}
    du(0, x) = \E\left[\tilde\Theta_{0,t}^{-1}\, du(t, X_t)\,
      \Theta_{0,t}\right]
  \end{equation*}
  for every $t\geq 0$. The curvature conditions imply that
  $\|\Theta_{0,t}\|\leq e^{-\alpha t/2}$ and $\|\tilde
  \Theta_{0,t}^{-1}\|\leq 1$, so that
  \begin{equation*}
    |du(0,x)|\leq e^{-\alpha t/2}\sup_{y\in M} |du(t, y)|.
  \end{equation*}
  The claim now follows from letting $t\to\infty$.
\end{proof}

To prove Liouville theorems under the weaker assumption that the
metric on $M$ evolves under backward super Ricci flow (not necessarily
uniformly strict backward super Ricci flow) one needs more refined
representation formulae which rely on integration by parts arguments.

For $X$ and $\X$ as above, let $B = (B_t)_{t\geq 0}$ be the
Riemann-anti-development of $X$ into $T_xM$ (hence a $T_xM$-valued
Brownian motion, see \cite[Remark~8.4]{martingales}), and
\begin{align}
  \label{Eq:AntiDef}
  \A_{\defo}(\X)_t := \int_0^t\tilde\Theta_{0,s}^{-1}\circ d\X_s =
  \int_0^t\tilde\Theta_{0,s}^{-1}\, du(s,X_s)\,\parr_{0,s}^{\Riem}\,
  dB_s
\end{align}
the deformed anti-development of $\X$
(cf.~\cite[Eq.~(5.32)]{arnaudonthalmaier}).

\begin{theorem}[cf.\ {\cite[Theorem~5.6]{arnaudonthalmaier}} for the case of a fixed metric]\label{localmart}
  Let $\ell = (\ell(t))_{t\geq 0}$ be a $T_xM$-valued process with
  absolutely continuous trajectories. The following
  $T_{u(0,x)}N$-valued processes are local martingales:
  \begin{align*}
    N_t & :=  \tilde\Theta_{0,t}^{-1}\, du(t, X_t)\,\Theta_{0,t}\,\ell(t) - \int_0^t\tilde\Theta_{0,s}^{-1}\,du(s, X_s)\,\Theta_{0,s}\,\dot\ell(s)\,ds,\\
    M_t & := \tilde\Theta_{0,t}^{-1}\, du(t,
    X_t)\,\Theta_{0,t}\,\ell(t) -
    \int_0^t\left(\parr_{0,s}^{\Riem}\right)^{-1}\,\Theta_{0,s}\,\dot\ell(s)\cdot
    dB_s\; \A_{\defo}(\X)_t .
  \end{align*}
\end{theorem}

\begin{proof}
  We have
  \begin{equation*}
    d \left( \tilde\Theta_{0,t}^{-1} \, du(t, X_t) \, \Theta_{0,t} \, \ell(t) \right) = \tilde\Theta_{0,t}^{-1}\, du(t, X_t) \,
    \Theta_{0,t} \, \dot \ell(t) \, dt + d \left( \tilde\Theta_{0,t}^{-1} \, du(t, X_t) \, \Theta_{0,t} \right) \! \ell(t),
  \end{equation*}
  so that $N$ is a local martingale by Theorem~\ref{martheat}. Since
  the quadratic covariation of $\A_{\defo}(\X)_t$ and $\int_0^t
  \left(\parr_{0,s}^{\Riem}\right)^{-1}\Theta_{0,s}\dot\ell(s) \cdot
  dB_s$ equals
  \begin{equation*}
    \Bigl\langle\A_{\defo}(\X),\int_0^\bull \left(\parr_{0,s}^{\Riem}\right)^{-1}\Theta_{0,s}\,\dot\ell(s) \cdot dB_s\Bigr\rangle_t = \int_0^t\tilde\Theta_{0,s}^{-1}\, du(s,X_s)\,\Theta_{0,s}\,\dot\ell(s)\,ds,
  \end{equation*}
  it follows that $M$ is a local martingale as well.
\end{proof}

In general, the processes $N$ and $M$ of Theorem~\ref{localmart} are
only local martingales, not necessarily true martingales. To obtain
stochastic representation formulas by taking expectations, a possible
strategy is to stop these processes before the underlying Brownian
motion $X$ leaves a relatively compact domain. The $T_xM$-valued
process $\ell$ may then be chosen appropriately.

\begin{theorem}[cf.\ {\cite[Theorem~3.1, Remark~3.4 and Theorem~4.1]{thalmaierwang}}]
  Let $v\in T_xM$, $R > 0$,
  \begin{equation*}
    D_R := \{ (t,y)\in\R_+ \times M\mid d_{g(t)}(x,y) < R\}
  \end{equation*}
  and $\tau$ a bounded stopping time satisfying $\tau\leq\tau_R$,
  where
  \begin{align*}
    \tau_R  := \inf &\{ t\geq 0 \mid d_{g(t)}(x, X_t) \geq R \}\\
    = \inf &\{ t\geq 0 \mid (t, X_t) \notin D_R \}.
  \end{align*}
  Suppose that the process $\ell$ satisfies $\ell(0) = v$, $\ell(\tau)
  = 0$ and
  \begin{equation}\label{Eq:l-integrability}
    \E \left[ \left( \int_0^{\tau} |\dot \ell(s)|^2 ds \right)^{(1+\eps)/{2}} \right] < \infty
  \end{equation}
  for some $\eps > 0$. Then the following stochastic representation
  formulas hold:
  \begin{equation}\label{repformula}
    du(0, x)v = - \E \left[ \int_0^{\tau} \left(\parr_{0,s}^{\Riem}\right)^{-1}\Theta_{0,s} \,\dot\ell(s) \cdot dB_s \; \A_{\defo}(\X)_\tau\right]
  \end{equation}
  and
  \begin{equation*}
    du(0, x)v = - \E \left[ \int_0^{\tau}\tilde\Theta_{0,s}^{-1}\,du(s, X_s)\,\Theta_{0,s}\,\dot\ell(s)\,ds\right].
  \end{equation*}
\end{theorem}

\begin{proof}
  To prove the claimed representation formulas it is sufficient to
  show that the stopped processes $(M_{t \wedge \tau})_{t \geq 0}$ and
  $(N_{t \wedge \tau})_{t \geq 0}$ are true martingales. By
  Theorem~\ref{localmart} we already know that they are local
  martingales.  To show that the process $(M_{t \wedge \tau})_{t
    \geq0}$ is a true martingale it suffices to show that for each $c
  \geq 0$ the family $\{M_{\sigma \wedge \tau} \mid \sigma \mbox{
    stopping time, } \sigma \leq c\}$ is uniformly integrable, which
  holds if there is a constant $C < \infty$ such that for every
  stopping time $\sigma \leq c$
  \begin{equation}\label{uniin}
    \E \left[ \left| M_{\sigma \wedge \tau} \right|^{1+\eps} \right] \leq C.
  \end{equation}
  We first observe that the terms $\tilde \Theta_{0,t}^{-1}\, du(t,
  X_t)\, \Theta_{0,t} \, \dot\ell(t)$ and $\A_{\defo}(\X)_t$ are
  bounded as long as $t \leq {\tau_R \wedge c}$. Moreover, using the
  Burkholder-Davis-Gundy inequality (see e.g.\
  \cite[Theorem~3.3.28]{karatzas}) and the fact that $\Theta_{0,s}$ is
  bounded for $s \leq \tau_R \wedge c$, we obtain
  \begin{equation*}
    \E \left[ \left| \int_0^{\sigma \wedge \tau} \left(\parr_{0,s}^{\Riem}\right)^{-1} \Theta_{0,s} \dot \ell(s) \cdot dB_s \right|^{1+\eps} \right] \leq \tilde C \,\E \left[ \left( \int_0^{\tau} |\dot \ell(s)|^2 ds \right)^{(1+\eps)/{2}} \right]
  \end{equation*}
  and hence \eqref{uniin}, so that the process $(M_{t \wedge \tau})_{t
    \geq 0}$ is indeed a martingale.

  In a similar way the process $(N_{t \wedge \tau})_{t \geq 0}$ is
  shown to be a martingale as well.
\end{proof}

\section{A priori estimates}
In this section we prove differential estimates for space-time
harmonic maps $u\colon\R_+\times M\to N$ under the assumption that the
metric on $M$ evolves under backward super Ricci flow
\begin{equation*}
  \frac{\partial g}{\partial t}  \leq \Ric_{g(t)}.
\end{equation*} 
By time reversal the results apply to ancient solutions to the
harmonic map heat flow under forward super Ricci flow. These estimates
will then be used in the next section to derive Liouville type results
for space-time harmonic mappings, respectively ancient solutions to
the harmonic map heat flow. The starting point of our approach is the
estimate
\begin{equation}\label{estimate}
  |du(0, x)v| \leq \E \left[ \left| \int_0^{t \wedge \tau_R} \left(\parr_{0,s}^{\Riem}\right)^{-1} \Theta_{0,s} \,\dot\ell(s) \cdot dB_s \right|^p \right]^{1/p} \E \left[ |\A_{\defo}(\X)_{t \wedge \tau_R}|^q \right]^{1/q}
\end{equation}
for $p,q > 1$ such that $1/p + 1/q = 1$, which follows immediately
from formula~\eqref{repformula}. The process~$\ell$ satisfies $\ell(0)
= v$, $\ell(\tau)= 0$ and condition~\eqref{Eq:l-integrability};
otherwise it may be chosen arbitrarily.

Note that in estimate \eqref{estimate} geometric information of the
evolving manifold $M$ only enters through the first term on the
right-hand side, while the second term captures the geometry of the
target~$N$. In this sense, estimate \eqref{estimate} allows to
separate the contributions of the curvatures of $M$ and $N$ to the
differential of $u$.

\subsection{Estimation of the first factor}\label{estfirstf}
To estimate the first factor on the right-hand side of
\eqref{estimate}, we have to choose the process $\ell$ in a suitable
way. To this end we fix $R > 0$ and, similarly to \cite[Proof of
Corollary~5.1]{thalmaierwang98}, define $f: \bar D_R \to [0,1]$ by
\begin{equation*}
  f(u,y) := \cos \left( \frac{\pi}{2R} \, d_{g(u)}(x,y) \right)\!.
\end{equation*}
For $p \geq 1$ let
\begin{equation*}
  c_p(R) := \sup_{(u,y) \in \tilde D_R} \left\{ f^{p+2} \left( \frac{\partial f^{-p}}{\partial u} + \frac{1}{2} \Delta_{g(u)}(f^{-p}) \right) \right\},
\end{equation*}
where $\tilde D_R := \{(t,y) \in D_R: y \neq x \mbox{ and } y \notin
\Cut_{g(t)}(x)\}$.

\begin{lemma}\label{cprlemma}
  Suppose that
  \begin{equation*}
    \frac{\partial g}{\partial t} \leq \Ric_{g(t)}
  \end{equation*}
  and that there exists $r_0 > 0$ such that
  \begin{equation}\label{cxr0}
    C(x, r_0) := \sup \left\{ |\Ric(t,y)|:\  t \geq 0, \, d_{g(t)}(x,y) \leq r_0 \right\}
  \end{equation}
  is finite. 
  Then $c_p(R)$ is finite for each $R > 0$, and moreover
  \begin{equation}\label{cpr}
    c_p(R) = O(1/R),\quad\text{as $R \to \infty$.}
  \end{equation}
\end{lemma}

\begin{remark}\label{remest}
  In the case of a fixed metric $g$ with non-negative Ricci curvature,
  instead of \eqref{cpr} one obtains the much better estimate
  \[
  c_p(R) \leq \frac{\pi^2 (d+3)}{4 R^2},
  \]
  see \cite[Proof of Corollary~5.1]{thalmaierwang98}. This is due to
  the fact that in the case of a fixed metric with non-negative Ricci
  curvature the estimate for the drift of the radial part of Brownian
  motion is much better than in the case of backward super Ricci flow,
  see Remark~\ref{driftbr}
  below. 
\end{remark}

\begin{proof}[Proof of Lemma~\rm\ref{cprlemma}]
  We first observe that
  \begin{equation*}
    c_p(R) = \sup_{(u,y) \in \tilde D_R} \left\{ \frac{p (p+1)}{2} \,|\nabla f|^2 - p f \left( \frac{\partial f}{\partial u} + \frac{1}{2} \Delta_{g(u)} f \right) \right\}.
  \end{equation*}
  Let now $\rho(u,y) := d_{g(u)}(x,y)$ and $\bar f(\xi) := \cos \big(
  \frac{\pi \xi}{2R} \big)$, so that $f(u,y) = \bar f(\rho(u,y))$ and
  consequently
  \begin{equation*}
    c_p(R) = \sup_{(u,y) \in \tilde D_R} \left\{ \frac{p (p+1)}{2} \,\bar f'(\rho)^2 \,|\nabla \rho|^2 - p \bar f(\rho) \left( \frac{1}{2} \bar f''(\rho) \,|\nabla \rho|^2 + \bar f'(\rho) \left( \frac{\partial \rho}{\partial u} + \frac{1}{2} \Delta_{g(u)} \rho \right) \right) \right\}.
  \end{equation*}
  Since $|\bar f| \leq 1$, $|\bar f'| \leq \pi/(2R)$, $|\bar f''| \leq
  \pi^2/(4R^2)$ and $|\nabla \rho| \equiv 1$ (on $\tilde D_R$), it
  follows that
  \begin{equation*}
    c_p(R) \leq \frac{(p^2 + 2p) \pi^2}{8 R^2} + \frac{p \pi}{2 R} \,\sup_{(u,y) \in \tilde D_R} \left\{ \sin \left( \frac{\pi \rho}{2 R} \right) \left( \frac{\partial \rho}{\partial u} + \frac{1}{2} \Delta_{g(u)} \rho \right) \right\}.
  \end{equation*}
  By \cite[Proposition~2]{kuwadaphilipowski} we have
  \begin{equation}\label{driftest}
    \frac{\partial \rho}{\partial u} + \frac{1}{2} \Delta_{g(u)} \rho \leq \frac{d-1}{2} \left( k(r_0) \coth(k(r_0) (\rho(u,y) \wedge
      r_0)) + k(r_0)^2 (\rho(u,y) \wedge r_0) \right)\!,
  \end{equation}
  where $k(r_0) := \sqrt{\frac{C(x, r_0)}{d-1}}$.  Therefore, using
  the inequality $\coth \xi \leq 1 + {1}/{\xi}$ valid for $\xi > 0$,
  we obtain for $\rho(u,y) \leq r_0$,
\begin{align*}
\frac{p \pi}{2 R} \sin \left( \frac{\pi \rho}{2 R} \right) \left( \frac{\partial \rho}{\partial u} + \frac{1}{2} \Delta_{g(u)} \rho \right) & \leq \frac{p \pi}{2 R} \sin \left( \frac{\pi \rho}{2 R} \right) \frac{d-1}{2} \left( k(r_0) \coth(k(r_0) \rho) + k(r_0)^2 \rho \right)\\
& \leq  \frac{(d-1) p \pi^2}{8R^2} \left( k(r_0) \rho + 1 + k(r_0)^2 \rho^2 \right)\\
& \leq \frac{(d-1)p\pi^2}{8R^2} \left( k(r_0) r_0 + 1 + k(r_0)^2 r_0^2 \right)\!,
\end{align*}
and for $\rho(u,y) \geq r_0$,
\begin{align*}
\frac{p \pi}{2 R} \sin \left( \frac{\pi \rho}{2 R} \right) \left( \frac{\partial \rho}{\partial u} + \frac{1}{2} \Delta_{g(u)} \rho \right) & \leq \frac{p \pi}{2 R} \sin \left( \frac{\pi \rho}{2 R} \right) \frac{d-1}{2} \left( k(r_0) \coth(k(r_0) r_0) + k(r_0)^2 r_0 \right)\\
& \leq \frac{(d-1)p\pi}{4 R} \left( k(r_0) + \frac{1}{r_0} + k(r_0)^2 r_0 \right)\!,
  \end{align*}
  which completes the proof.
\end{proof}

\begin{remark}\label{driftbr}
  The key ingredient of the proof above is estimate \eqref{driftest}
  for the radial drift of Brownian motion, which should be seen as a
  parabolic version of the Laplacian comparison theorem for evolving
  manifolds.  In the case of a fixed metric with non-negative Ricci
  curvature the Laplacian comparison theorem however provides the much
  better estimate
  \begin{equation}\label{lct}
    \frac{1}{2} \Delta \rho \leq \frac{d-1}{2 \rho}.
  \end{equation}
  Since in many respects manifolds evolving under backward super Ricci
  flow behave in a similar way as manifolds with a fixed metric of
  non-negative Ricci curvature (see e.g.\ \cite{mccanntopping} or
  \cite[Section~6.5]{toppinglectures}), one might expect that an
  estimate similar to \eqref{lct} also holds under backward super
  Ricci flow. This, however, is not the case, as the following example
  shows.
\end{remark}

\begin{example}[Brownian motion on Hamilton's cigar]
  Let $M = \R^2$ be equipped with the time-dependent metric
  \begin{equation*}
    g(t,x) := \frac{1}{e^{-2t} + |x|^2} \, g_{\,\text{eucl}}(x),
  \end{equation*}
  where $g_{\,\text{eucl}}$ denotes the standard metric on $\R^2$. As
  shown in \cite[Section~4.3]{chowluni}, the family $(g(t))_{t \in
    \R}$ is an eternal solution of the backward Ricci flow, called
  ``Hamilton's cigar'' or ``Witten's black hole''. By elementary
  calculations one obtains
  \begin{equation*}
    \rho(t,x)  = \arcsinh(e^t |x|),
  \end{equation*}
  and consequently
  \begin{equation*}
    \frac{\partial \rho}{\partial t}(t,x) = \frac{1}{\sqrt{1 + e^{-2t} |x|^{-2}}}
  \end{equation*}
  and
  \begin{equation*}
    \Delta_{g(t)} \rho(t,x) = \frac{1}{e^{2t} |x|^2 \sqrt{1 + e^{-2t} |x|^{-2}}}.
  \end{equation*}
  It is now easy to see that for each $t \in \R$ the function $|x|
  \mapsto \left( \frac{\partial \rho}{\partial t} + \frac{1}{2}
    \Delta_{g(t)} \rho \right) (t,x)$ is decreasing, and hence bounded
  from below by
  \begin{equation*}
    \lim_{|x| \to \infty} \left( \frac{\partial \rho}{\partial t} 
    + \frac{1}{2} \Delta_{g(t)} \rho \right) (t,x) = 1.
  \end{equation*}
Consequently, a parabolic analogue 
to \eqref{lct} cannot hold under backward super Ricci flow. The drift part of the distance process $\rho(t,X_t)$ of Brownian motion grows at least like $t$.
\end{example}

We now consider the strictly increasing process $(T(s))_{s \in [0,
  \tau_R)}$ given by
\[
T(s) := \int_0^s \frac1{f^2(u, X_u)} \, du,
\]
and let the process $(\sigma(r))_{r \geq 0}$ be defined by
\[
\sigma(r) :=
\begin{cases}
  \inf \left\{ s \in [0, \tau_R)\colon\, T(s) \geq r \right\} & \mbox{if such an $s$ exists,}\\
  \tau_R & \mbox{otherwise.}
\end{cases}
\]

\begin{lemma}
  For all stopping times $\tau \leq \tau_R$ we have
  \[
  \E \left[ f^{-p}(\tau, X_\tau) \right] \leq \E \left[ e^{c_p(R)
      T(\tau)} \right].
  \]
\end{lemma}

\begin{proof}
  Applying It\^o's formula to the process $Y_r := f^{-p}(\sigma(r),
  X_{\sigma(r)})$ ($0 \leq r < T(\tau_R)$) we obtain
  \begin{align*}
    dY_r & \stackrel{\m}{\leq} \left( \frac{\partial f^{-p}}{\partial u} + \frac{1}{2} \Delta_{g(\sigma(r))}(f^{-p}) \right) (\sigma(r), X_{\sigma(r)}) \dot \sigma(r) \, dr\\
    & = f^{-p}(\sigma(r), X_{\sigma(r)}) \left[ f^{p+2} \left( \frac{\partial f^{-p}}{\partial u} + \frac{1}{2} \Delta_{g(\sigma(r))}(f^{-p}) \right) \right] (\sigma(r), X_{\sigma(r)}) \, dr\\
    & \leq c_p(R)\, f^{-p}(\sigma(r), X_{\sigma(r)}) \, dr\\
    & = c_p(R)\, Y_r \, dr,
  \end{align*}
  where the inequality (modulo differentials of local martingales) in the first step is due to the local time at
  the cut-locus, see
  \cite[Theorem~2]{kuwadaphilipowski}. 
  Since $Y_0 = 1$, it follows that
  \[
  \E \left[ f^{-p}(\tau, X_\tau) \right] = \E \! \left[ Y_{T(\tau)}
  \right] \leq \E \left[ e^{c_p(R) T(\tau)} \right]. \qedhere
  \]
\end{proof}

\begin{lemma}
  If $c_p(R)$ is finite, we have
  \[
  \lim_{s \uparrow \tau_R} T(s) = +\infty
  \]
  almost surely.
\end{lemma}

\begin{proof}
  Let $\tau^n := \inf \{ t \geq 0\colon\, f(t,X_t) \leq 1/n\}$. The
  previous lemma with $p=1$ and $\tau = \sigma(t) \wedge \tau^n$
  implies that for each $t \geq 0$
  \[
  n \PP \left\{ \tau^n \leq \sigma(t) \right\} \leq \E \left[
    f^{-1}(\sigma(t) \wedge \tau^n, X_{\sigma(t) \wedge \tau^n})
  \right] \leq e^{c_1(R) t}
  \]
  and consequently, since $\tau^n \uparrow \tau_R$,
  \[
  \PP \left\{ \sigma(t) = \tau_R \right\} = \lim_{n \to \infty} \PP
  \left\{ \tau^n \leq \sigma(t) \right\} = 0.
  \]
  Since
  \[
  \left\{ \lim_{s \uparrow\tau_R} T(s) < \infty\right\} =
  \bigcup_{0\leq t\in\Q} \{ \sigma(t) = \tau_R \},
  \]
  the claim follows.
\end{proof}

\begin{lemma}[cf.\ {\cite[Lemma~4.3]{thalmaierwang}} for the case of a fixed metric]\label{firstfactor}
  Assume that
  \begin{equation*}
    \frac{\partial g}{\partial t} \leq \Ric_{g(t)}.
  \end{equation*}
  and that there exists $r_0 < 0$ such that $C(x,r_0)$ defined in
  \eqref{cxr0} is finite. Then, for all $t \geq 0$,
  \begin{equation}\label{estm}
    \E \left[ \left| \int_0^{t \wedge \tau_R} \left( \parr_{0,s}^{\Riem} \right)^{-1} \Theta_{0,s} \, \dot\ell(s) \cdot dB_s \right|^p \right] \leq \frac{C_p \left(2 c_p(R) / p \right)^{p/2}}{\left(1 - \exp(-2c_p(R) t/p)\right)^{p/2+1}} |v|^p.
  \end{equation}
  where $C_p$ is the constant in the Burkholder-Davis-Gundy inequality
  with exponent $p$.
\end{lemma}

\begin{proof}
  Fix $t > 0$. As a consequence of the previous lemma we have
  \[
  \sigma(r) = \inf \{ s \in [0, \tau_R): T(s) \geq r \}.
  \]
  Moreover, $\sigma(r) \leq r$, $\sigma(r) \leq \tau_R$, $T(\sigma(r))
  = r$ and
  \[
  \dot \sigma(r) = {1}/{\dot T(\sigma(r))} = f^2(\sigma(r),
  X_{\sigma(r)})
  \]
  for all $r \geq 0$, and $\sigma(T(s)) = s$ for all $s \in [0,
  \tau_R)$. Now let
  \begin{equation*}
    h_0(s) := \int_0^{s \wedge \sigma(t)} \frac{1}{f^2(r, X_r)} \, dr = T(s \wedge \sigma(t)) = T(s) \wedge t
  \end{equation*}
  and
  \begin{equation*}
    h_1(r) = 1 - \frac{1-\exp(-2c_p(R) r/p)}{1 - \exp(-2c_p(R) t/p)},
  \end{equation*}
  and define
  \begin{equation*}
    \ell(s) := h_1(h_0(s))\, v.
  \end{equation*}
  Note that $h_1(0) = 1$, $h_1(t) = 0$ and $\dot h_1(r) < 0$ for all
  $r \geq 0$, so that $|\dot h_1(r)|\,dr$ is a probability measure on
  $[0, t]$.

  Since $\frac{\partial g}{\partial t} \leq \Ric$ implies that
  $|\Theta_{0,s}| \leq 1$, we obtain using the Burkholder-Davis-Gundy
  inequality that
  \[
  \E \left[ \left| \int_0^{t \wedge \tau_R}
      \left(\parr_{0,s}^{\Riem}\right)^{-1}\Theta_{0,s}\,\dot\ell(s)
      \cdot dB_s \right|^p \right] \leq C_p \, \E \left[ \left|
      \int_0^{t \wedge \tau_R} \bigl| \dot \ell(s) \bigr|^2\, ds
    \right|^{p/2} \right].
  \]
  Moreover, since
  \begin{equation*}
    \dot h_0(s) = 
    \begin{cases}
      f^{-2}(s, X_s) & \mbox{if } s < \sigma(t),\\
      0 & \mbox{if } s > \sigma(t),
    \end{cases}
  \end{equation*}
  we have
  \begin{align*}
    \E \left[ \left| \int_0^{t \wedge \tau_R} \bigl| \dot \ell(s) \bigr|^2\, ds \right|^{p/2} \right] & = \E \left[ \left| \int_0^{t \wedge \tau_R} |\dot h_1(h_0(s))|^2\, |\dot h_0(s)|^2\, ds \right|^{p/2} \right] |v|^p\\
    & = \E \left[ \left| \int_0^{\sigma(t)} |\dot h_1(h_0(s))|^2\, \frac{1}{f^4(s, X_s)}\, ds \right|^{p/2} \right] |v|^p\\
    & = \E \left[ \left| \int_0^t |\dot h_1(h_0(\sigma(r)))|^2 \,\frac{1}{f^4(\sigma(r), X_{\sigma(r)})} \dot \sigma(r)\,dr \right|^{p/2} \right] |v|^p\\
    & = \E \left[ \left| \int_0^t |\dot h_1(r)|^2 \,\frac{1}{f^2(\sigma(r), X_{\sigma(r)})} \,dr \right|^{p/2} \right] |v|^p\\
    & \leq \E \left[ \int_0^t |\dot h_1(r)|^{p/2+1} \,\frac{1}{f^p(\sigma(r), X_{\sigma(r)})} \,dr \right] |v|^p\\
    & = \int_0^t |\dot h_1(r)|^{p/2+1} \,\E \left[ \frac{1}{f^p(\sigma(r), X_{\sigma(r)})} \right] dr \; |v|^p\\
    & \leq |v|^p \int_0^t |\dot h_1(r)|^{p/2+1} \,e^{c(R,f)r} \, dr,
  \end{align*}
  where in the fifth step we used Jensen's inequality with respect to
  the probability measure $|\dot h_1(r)|\,dr$ on $[0,t]$ (recall that
  $p \geq 2$).

  Since
  \begin{equation*}
    \dot h_1(r) = -\frac{2 c_p(R) / p}{1 - \exp(-2c_p(R) t/p)} \, e^{- 2c_p(R)r/p}
  \end{equation*}
  and consequently
  \begin{equation*}
    |\dot h_1(r)|^{p/2+1} \,e^{c_p(R)r} = \left( \frac{2 c_p(R) / p}{1 - \exp(-2c_p(R) t/p)} \right)^{p/2+1} e^{- 2c_p(R)r/p},
  \end{equation*}
  the claim follows.
\end{proof}

\subsection{Estimation of the second factor}
To estimate the second factor of \eqref{estimate} we start by
estimating the inverse of the damped parallel transport.

\begin{lemma}[cf.\ {\cite[Lemma~2.12]{thalmaierwang}}]\label{invdamp}
  For $y \in N$ let $\kappa(y)$ be the supremum of the sectional
  curvatures of $N$ at $y$. Moreover, for $(s, x) \in \R_+ \times M$,
  let $\lambda_1(s, x) \geq \ldots \geq \lambda_m(s, x) \geq 0$ be the
  eigenvalues of the map $du(s, x)^* du(s, x): T_x M \to T_x M$. Then
  \begin{equation*}
    \bigl\| \tilde \Theta_{0,t}^{-1} \bigr\| \leq \exp \left( \frac{1}{2} \int_0^t L_s\,ds \right),
  \end{equation*}
  where
  \begin{equation*}
    L_s := 
    \begin{cases}
      |du|^2(s,X_s)\, \kappa(u(s,X_s)) & \mbox{if } \kappa(u(s,X_s)) \geq 0\\
      \sum\limits_{i=2}^m \lambda_i(s,X_s) \, \kappa(u(s,X_s)) &
      \mbox{if } \kappa(u(s,X_s)) \leq 0.
    \end{cases}
  \end{equation*}
\end{lemma}

\begin{proof}
  Let $w \in T_{\tilde X_0} N$. Using the definitions of $\tilde
  \Theta_{0,t}$ and $\kappa$ we obtain on $\{\kappa(u(t,X_t)) \geq
  0\}$
  \begin{align*}
    \frac{d}{dt} \big| \tilde \Theta_{0,t} w \bigr|^2 & = -\sum_{i=1}^m \left\langle \tilde R\bigl(\tilde\Theta_{0,t} w, du(t, X_t)\,\xi_t^i \bigr) du(t, X_t)\,\xi_t^i,\tilde\Theta_{0,t} w\right\rangle \\
    & \geq -\sum_{i=1}^m \kappa(u(t,X_t)) \left[ \bigl| \tilde \Theta_{0,t} w \bigr|^2 \left| du(t, X_t)\,\xi_t^i \right|^2 - \left\langle \tilde\Theta_{0,t} w, du(t, X_t) \,\xi_t^i \right\rangle^2 \right]\\
    & \geq  -\sum_{i=1}^m \kappa(u(t,X_t))\, \bigl| \tilde \Theta_{0,t} w \bigr|^2 \,\left| du(t, X_t) \,\xi_t^i \right|^2\\
    & =  -\kappa(u(t,X_t)) \,\bigl| \tilde \Theta_{0,t} w \bigr|^2  \,|du|^2(t, X_t)\\
    & = -L_t\, \bigl| \tilde \Theta_{0,t} w \bigr|^2.
  \end{align*}
  On the set $\{\kappa(u(t,X_t)) \leq 0\}$ we moreover use the fact
  that
  \[
  \sum_{i=1}^m \left\langle \tilde\Theta_{0,t} w, du(t, X_t) \xi_t^i
  \right\rangle^2 \leq \lambda_1(t, X_t) \,|\tilde \Theta_{0,t} w|^2
  \]
  and obtain
  \begin{align*}
    \frac{d}{dt} \bigl| \tilde \Theta_{0,t} w \bigr|^2 & \geq  -\sum_{i=1}^m \kappa(u(t,X_t)) \left[ \bigl| \tilde \Theta_{0,t} w \bigr|^2 \left| du(t, X_t) \xi_t^i \right|^2 - \left\langle\tilde\Theta_{0,t} w, du(t, X_t) \xi_t^i \right\rangle^2 \right]\\
    & \geq  -\kappa(u(t,X_t))\, \bigl| \tilde \Theta_{0,t} w \bigr|^2 \,\sum_{i=2}^m \lambda_i (t, X_t)\\
    & = -L_t\, \bigl| \tilde \Theta_{0,t} w \bigr|^2
  \end{align*}
  as well.
\end{proof}

\begin{lemma}[cf.\ {\cite[Lemma~4.5~(1)]{thalmaierwang}}]\label{sf}
  For any stopping time $\tau$ we have
  \[
  \E \left[ |\A_{\defo}(\X)_\tau|^q \right] \leq C_q \, \E \left[
    \left( \int_0^{\tau} |du|^2(s, X_s) \exp \left( \int_0^s
        |du|^2(r,X_r) \, \kappa_+(u(r,X_r)) \,dr \right) ds
    \right)^{q/2} \right], 
  \]
  where $\kappa_+(y) := \max(\kappa(y),0)$.
\end{lemma}

\begin{proof}
  By the Burkholder-Davis-Gundy inequality we have
  \[
  \E \left[ |\A_{\defo}(\X)_\tau|^q \right] \leq C_q \E \left[ \left|
      \int_0^{\tau} \bigl| \tilde \Theta_{0,s}^{-1} \, du(s, X_s)
      \bigr|^2 ds \right|^{q/2} \right],
  \]
  and by the previous lemma
  \begin{equation*}
    \bigl\| \tilde \Theta_{0,s}^{-1} \bigr\|^2 \leq \exp \left( \int_0^s |du|^2(r, X_r) \,\kappa_+(u(r, X_r)) \,dr \right). \qedhere
  \end{equation*}
\end{proof}

\begin{lemma}[cf.\ {\cite[Lemma~4.5~(2)]{thalmaierwang}}]\label{ch}
  If $N$ is simply connected and has non-positive curvature, then for
  any bounded stopping time $\tau$
  \[
  \E \left[ |\A_{\defo}(\X)_\tau|^2 \right] \leq \E \left[
    \dist_N\bigl(u(\tau, X_{\tau}), u(0,x)\bigr)^2 \right].
  \]
\end{lemma}

\begin{proof}
  Let $\rho(z) := \dist_N(z,
  u(0,x))$. 
  Using the chain rule for the tension field (see e.g.\
  \cite[Lemma~8.7.2]{jost}) and the Hessian comparison theorem (see
  e.g.\ \cite[Satz~7.236]{hackenbroch}) we obtain
  \begin{align*}
    -\frac{\partial (\rho^2 \circ u)}{\partial t} + \frac{1}{2} \Delta(\rho^2 \circ u) & =  -\nabla(\rho^2) \frac{\partial u}{\partial t} + \frac{1}{2} \tr \! \left( \Hess(\rho^2) \circ (du \otimes du) \right) + \frac{1}{2} \nabla(\rho^2) \Delta u\\
    & \geq |d u|^2,
  \end{align*}
  and therefore, using the Burkholder-Davis-Gundy inequality and
  It\^o's formula
  \begin{align*}
    \E \left[ |\A_{\defo}(\X)_\tau|^2 \right] & \leq     \E \left[ \int_0^{\tau} \bigl| \tilde \Theta_{0,s}^{-1} du(s, X_s) \bigr|^2 ds \right]\\
    & \leq  \E \left[ \int_0^{\tau} |du|^2(s, X_s)\,ds \right]\\
    & \leq  \E \left[ \int_0^{\tau} \left( -\frac{\partial (\rho^2 \circ u)}{\partial t} + \frac{1}{2} \Delta(\rho^2 \circ u) \right) (s, X_s) \,ds \right]\\
    & = \E \left[ \rho^2(u(\tau, X_{\tau})) \right],
  \end{align*}
  as claimed.
\end{proof}

\begin{lemma}[cf.\ {\cite[Lemma~4.6]{thalmaierwang}}]
  Assume that $N$ has non-positive curvature and let
  \begin{equation*}
    K(t, x) := \frac{\lambda_1(t,x)}{\sum\limits_{i=2}^m \lambda_i(t,x)}
  \end{equation*}
  (with the convention ${0}/{0} := 0$).
  Then for any stopping time $\tau$
  \[
  \E \left[ |\A_{\defo}(\X)_\tau|^q \right] \leq C_q \E \left[ \left|
      \int_0^{\tau} |du|^2(s, X_s) \exp \left( \int_0^s |du|^2(r,X_r)
        \,\frac{\kappa(u(s,X_s))}{K(s, X_s)} \,dr \right) ds
    \right|^{q/2} \right].
  \]
\end{lemma}

\begin{proof}
  This follows immediately from the Burkholder-Davis-Gundy inequality,
  Lemma~\ref{invdamp} and the definition of $K$.
\end{proof}

\begin{corollary}\label{bdc}
  Suppose that
  \[
  \frac{\kappa(u(s,x))}{K(s,x)} \leq -b < 0
  \]
  for all $s \geq 0$ and all $x \in M$. Then we have
  \[
  \E \left[ |\A_{\defo}(\X)_\tau|^2 \right] \leq \frac{1}{b} \, \E
  \left[ \left( 1 - \exp \left( -b \int_0^{\tau} |du|^2(r, X_r) \,dr
      \right) \right) \right] \leq \frac{1}{b}.
  \]
\end{corollary}

\begin{proof}
  Since
  \begin{equation*}
    \frac{d}{ds} \exp \left( -b \int_0^s |du|^2(r,X_r) \,dr \right) = -b\, |du|^2(s, X_s) \exp \left( -b \int_0^s |du|^2(r,X_r) \,dr \right)
  \end{equation*}
  we have
  \begin{align*}
    \int_0^{\tau} |du|^2(s, X_s) \exp &\left( \int_0^s |du|^2(r, X_r) \frac{\kappa(u(s, X_s))}{K(s, X_s)} \, dr \right) ds\\
    &\leq \frac{1}{b} - \frac{1}{b} \exp \left( -b \int_0^{\tau}
      |du|^2(r, X_r) \,dr \right). \qedhere
  \end{align*}
\end{proof}

\section{Liouville theorems}
In this section we derive Liouville type results for space-time
harmonic mappings, respectively ancient solutions to the harmonic map
heat flow. We work out details in three typical cases: Mappings of
sub-square-root growth, of bounded dilatation, and of small
image. From now on we suppose that $M$ is connected.

\subsection{Space-time harmonic maps of sub-square-root growth}
We say that a function $u: \R_+\times M \to N$ is of
\emph{sub-square-root growth} if for each $x \in M$ there exists a
function $\varphi: \R_+ \to \R_+$ with
\begin{equation}\label{ssqdef}
  \varphi(r)/\sqrt{r} \to 0
\end{equation}
as $r \to \infty$ such that for all $t\geq0$ and all $z\in M$
\begin{equation}\label{sublinear}
  \dist_N \! \left( u(t,z), u(0,x) \right) \leq \varphi(d_{g(t)} (z,x)).
\end{equation}

\begin{theorem}
  \label{subwurzel}
  Suppose that $M$ is connected,
  \begin{equation*}
    \frac{\partial g}{\partial t} \leq \Ric_{g(t)}
  \end{equation*}
  (backward super Ricci flow), that for each $x \in M$ there exists
  $r_0 > 0$ such that the constant $C(x, r_0)$ defined in \eqref{cxr0}
  is finite, and that $N$ is simply-connected and has non-positive
  sectional curvatures. Then every space-time harmonic mapping
  $u\colon \R_+\times M \to N$ of sub-square-root growth is constant.
\end{theorem}

\begin{proof}
  By Lemma~\ref{firstfactor}, Lemma~\ref{ch} and \eqref{sublinear} we
  have for each $R > 0$
  \begin{align*}
    |du(0,x) v|^2 & \leq  c(R)\, |v|^2 \,\E \left[ \dist_N(u( \tau_R, X_{\tau_R}), u(0,x))^2 \right]\\
    & \leq  c(R)\, |v|^2 \,\E \left[ \varphi^2(d_{g(\tau_R)}(X_{\tau_R}, x)) \right]\\
    & = c(R)\, |v|^2 \,\varphi(R)^2.
  \end{align*}
  The claim now follows by letting $R \to \infty$, taking into account
  Lemma~\ref{cprlemma} and \eqref{ssqdef}.
\end{proof}

Analogously, a mapping $u\colon (-\infty,T]\times M \to N$ is said to
be of sub-square-root growth if \eqref{sublinear} holds for all $(t,
z) \in (-\infty,T] \times M$.

\begin{corollary}
  Suppose that $M$ is connected and that
  \begin{equation*}
    \frac{\partial g}{\partial t} \geq -\Ric_{g(t)}\quad\text{on\/ $(-\infty,T]\times M$}
  \end{equation*}
  (forward super Ricci flow). Assume that for each $x \in M$ there
  exists $r_0 > 0$ such that $C(x, r_0)$ as defined in \eqref{cxr0} is
  finite, and that $N$ is simply-connected and has non-positive
  sectional curvatures. Then any ancient solution of sub-square-root
  growth $u\colon (-\infty,T]\times M \to N$ to the harmonic map heat
  flow 
  is constant.
\end{corollary}

\begin{remark}
  Theorem~\ref{subwurzel} should be compared with S.-Y.~Cheng's
  Liouville theorem \cite{Cheng:80} which gives an analogous statement
  for harmonic maps of sublinear growth when $M$ is equipped with a
  fixed metric of non-negative Ricci curvature, see
  \cite[Corollary~5.10]{thalmaierwang} and \cite{Stafford:90} for
  stochastic proofs.  All these proofs depend crucially on the
  Laplacian comparison theorem.  In the case of backward super Ricci
  flow the stronger assumption of sub-square-root growth is needed
  because estimate \eqref{cpr} is weaker than estimate~\eqref{lct} in
  the case of a fixed metric with non-negative Ricci curvature, see
  the discussion in Remark~\ref{remest}.
\end{remark}

\subsection{Space-time harmonic maps of bounded dilatation}
Let $u\colon \R_+\times M \to N$ be a space-time harmonic map. We say
that $u$ is of {\em bounded dilatation} if there is a real constant
$C$ such that
\begin{equation}\label{Eq:Boundeddilatation}
  \lambda_1(t,x) \leq C \sum_{i=2}^m \lambda_i(t,x)
\end{equation}
for all $(t, x) \in \R_+\times M$. Similarly, an ancient solution
$u\colon (-\infty,T]\times M \to N$ to the harmonic map heat flow is
said to be of bounded dilatation if \eqref{Eq:Boundeddilatation} holds
for all $(t, x) \in (-\infty,T]\times M$.

\begin{theorem}[cf. {\cite[Corollary~5.15]{thalmaierwang}} for harmonic maps in the case of a fixed metric]
  Suppose that $M$ is connected,
  \begin{equation*}
    \frac{\partial g}{\partial t} \leq \Ric_{g(t)}
  \end{equation*}
  (backward super Ricci flow), that for each $x \in M$ there exists
  $r_0 > 0$ such that $C(x, r_0)$ is finite, and that $N$ has
  uniformly strictly negative sectional
  curvatures. 
  Then any space-time harmonic map $u\colon \R_+\times M \to N$ of
  bounded dilatation is constant.
\end{theorem}

\begin{proof}
  The assumptions on the curvature of $N$ and the dilatation of $u$
  imply that there is a constant $b > 0$ such that
  \[
  \frac{\kappa(u(t,x))}{K^2(t,x)} \leq -b < 0
  \]
  for all $t\geq0$ and all $x \in M$. Lemma~\ref{firstfactor} and
  Corollary~\ref{bdc} then imply that
  \begin{equation*}
    |du(0,x) v|^2 \leq \frac{C_2 c(R) |v|^2}{b(1-\exp(-c(R) t))^2}.
  \end{equation*}
  Letting first $t \to \infty$ and then $R \to \infty$, one obtains
  that $du(0,x) = 0$.
\end{proof}

\begin{corollary}
  Suppose that $M$ is connected,
  \begin{equation*}
    \frac{\partial g}{\partial t} \geq -\Ric_{g(t)}\quad\text{on $(-\infty,T]\times M$}
  \end{equation*}
  (forward super Ricci flow), that for each $x \in M$ there exists
  $r_0 > 0$ such that $C(x, r_0)$ (defined in \eqref{cxr0}) is finite,
  and that $N$ has uniformly strictly negative sectional
  curvatures. Then any ancient solution of bounded dilatation $u\colon
  (-\infty,T]\times M \to N$ to the harmonic map heat flow is
  constant.
\end{corollary}

\subsection{Space-time harmonic maps of small image}
Let $(N,h)$ be a Riemannian manifold and $\lambda>0$. Recall that an
$N$-valued martingale $Y$ is said to have \emph{exponential moments of
  order} $\lambda$ if
\begin{equation*}
  \E\left[\exp\Bigl(\lambda\int_0^\infty h(dY_s,dY_s)\Bigr)\right]<\infty.
\end{equation*}

\begin{remark}\label{R:ExpMoments}
  Let $(N,h)$ be a Riemannian manifold, $B\subset N$ an open subset
  and $\lambda>0$.  Suppose that there is a real-valued $C^2$ function
  $f$ on $B$ satisfying $c_1\leq f\leq c_2$ for some positive
  constants $c_1$, $c_2$ such that
  \begin{equation}\label{Eq:HessEst}
    \nabla df+2\lambda f\leq0.
  \end{equation}
  Then every $N$-valued martingale taking its values in $B$ has
  exponential moments of order $\lambda$, see
  \cite[Proposition~2.1.2]{Picard:91},
  \cite[Remark~5.2]{thalmaierwang}.
\end{remark}

\begin{definition}
  Let $(N,h)$ be a Riemannian manifold, $y\in N$ a point, and
  $B=B(y,r)$ an open geodesic ball about $y$ of radius $r$. Such a
  geodesic ball is said to be \textit{regular} if it does not meet the
  cut locus of its center $y$ and if $\kappa<(\pi/2r)^2$ where
  $\kappa$ denotes an upper bound of the sectional curvatures of~$N$
  on $B(y,r)$.
\end{definition}

\begin{example}\label{Ex:GeodBall}
  Suppose that $B(y,r)$ is a relatively compact regular geodesic ball
  in $N$ such that $r<\pi/(2\sqrt\kappa)$ where $\kappa>0$ is an upper
  bound of the sectional curvatures of $N$ on $B(y,r)$. Let
  $f=\cos\left(\sqrt{\kappa q}\,d(y,\newdot)\right)$ where $q>1$ is
  chosen in such a way that $0<c_1\leq f$ holds on $B(y,r)$ for some
  $c_1>0$. Then
  \begin{equation*}
    \nabla df + \kappa q \, f \leq 0,
  \end{equation*}
  which by Remark \ref{R:ExpMoments} means that any $B(y,r)$-valued
  martingale has exponential moments of order $\kappa q/2$.
\end{example}

\begin{theorem}[cf. {\cite[Corollary~5.5]{thalmaierwang}} for harmonic maps in the case of a fixed metric]\label{C:SmallImage}
  Suppose that $M$ is connected and
  \begin{equation*}
    \frac{\partial g}{\partial t} \leq \Ric_{g(t)}
  \end{equation*}
  (backward super Ricci flow) and that for each $x \in M$ there exists
  $r_0 > 0$ such that $C(x, r_0)$ is finite. Let $B$ be a relatively
  compact regular geodesic ball in $N$ of radius $r$ such that $r <
  \pi / (2 \sqrt\kappa)$ where $\kappa>0$ is an upper bound of the
  sectional curvatures of $N$ on $B$.  Then any space-time harmonic
  map $u\colon \R_+\times M \to N$ taking its values in $B$ is
  constant.
\end{theorem}

\begin{proof}
  Let $q \in (1, 2]$ be as in Example \ref{Ex:GeodBall} and $p$ be
  such that $1/p + 1/q = 1$. Since by Lemma~\ref{firstfactor} the
  $L^p$-norm term on the right-hand side of \eqref{estimate} tends to
  $0$ as $t\to\infty$ and then $R\to\infty$, it is sufficient to show
  that $\E \bigl[ |\A_{\defo}(\X)_{t \wedge \tau_R}|^q \bigr]$ is
  bounded uniformly in $t$ and $R$. By Lemma~\ref{sf} we have
  \[
  \E \left[ |\A_{\defo}(\X)_{t \wedge \tau_R}|^q \right] \leq C_q \,
  \E \left[ \left( \int_0^{t \wedge \tau_R} \rho(s) \exp \left(
        \int_0^s \rho(r) \,\kappa_+(\tilde X_r) \,dr \right) ds
    \right)^{q/2} \right],
  \]
  where $\rho(s) := \vert du\vert^2(s,X_s)$. Since by assumption
  $\kappa_+(\tilde X_r)\leq\kappa$ for a constant $\kappa>0$, and
  using
  \begin{equation*}
    \rho(s) \exp \left\{ \kappa \int_0^s \rho(r) \, dr \right\} = \frac{1}{\kappa} \frac{d}{ds} \exp \left\{ \kappa \int_0^s \rho(r) \, dr \right\},
  \end{equation*}
  we finally obtain
  \begin{align*}
    \E \left[ |\A_{\defo}(\X)_{t \wedge \tau_R}|^q \right] & \leq C_q \E \left[ \left\{ \frac{1}{\kappa} \, \left( \exp \left\{ \kappa \int_0^\tau \vert du \vert^2(r, X_r) \, dr \right\} - 1 \right) \right\}^{q/2} \right]\\
    & \leq \frac{C_q}{\kappa^{q/2}} \E \left[ \left( \exp \left\{
          \kappa \int_0^\infty \vert du\vert^2(r, X_r) \, dr \right\}
        - 1 \right)^{q/2} \right].
  \end{align*}
  The last term is finite since by Example~\ref{Ex:GeodBall} the
  martingale $\tilde X_t=u(t,X_t)$ has exponential moments of order
  $\kappa q/2$, which is equivalent to say that
  \begin{equation*}
    \E\left[\exp\left(\frac{\kappa q}2\int_0^\infty\vert du\vert^2(r,X_r)\,dr\right) \right]<\infty.\qedhere
  \end{equation*}
\end{proof}

\begin{corollary}
  Suppose that $M$ is connected and
  \begin{equation*}
    \frac{\partial g}{\partial t} \geq -\Ric_{g(t)}\quad\text{on $(-\infty,T]\times M$}
  \end{equation*}
  (forward super Ricci flow). Let $B$ be a relatively compact regular
  geodesic ball in $N$ of radius $r$ such that $r < \pi / (2
  \sqrt\kappa)$ where $\kappa>0$ is an upper bound of the sectional
  curvatures of $N$ on $B$. Then any ancient solution $u \colon
  (-\infty,T]\times M \to N$ to the harmonic map heat flow taking
  values in $B$ is constant.
\end{corollary}

\end{document}